\documentclass[11pt]{article}
\usepackage{amsmath,amsthm,amssymb,amscd}
\usepackage{graphicx}
\usepackage{graphics}
\usepackage{verbatim}

\usepackage{mathrsfs}
\usepackage{color}
\usepackage[top=1in, bottom=1in, left=1.25in, right=1.25in]{geometry}
\usepackage{enumerate}

\newtheorem{theorem}{Theorem}[section]
\newtheorem{proposition}[theorem]{Proposition}

\newtheorem{lemma}[theorem]{Lemma}
\newtheorem{definition}[theorem]{Definition}

\newtheorem{claim}{Claim}[theorem]
\newtheorem{conjecture}[theorem]{Conjecture}

\begin{document}

\title{Approximating Vizing's independence number conjecture}


\author{Eckhard Steffen\thanks{
		Institute of Mathematics,
		Paderborn University,
		Warburger Str. 100,
		33098 Paderborn,
		Germany;	es@upb.de}}
\date{}

\maketitle

\begin{abstract}
{In 1965, Vizing conjectured that the independence ratio of edge-chromatic critical graphs is at most $\frac{1}{2}$. 
We prove that for every $\epsilon > 0$ this conjecture is equivalent to its restriction on a specific set of edge-chromatic critical graphs
with independence ratio smaller than $\frac{1}{2} + \epsilon$.}
\end{abstract}

\section{Introduction} \label{Intro}

All graphs in this article are simple. If $G$ is a graph, then $V(G)$ denotes its vertex set and $E(G)$ denotes its edge set. 
If $e \in E(G)$ has end vertices $v$ and $w$, then we also use the term $vw$ to denote $e$.
If $v$ is a vertex of $G$, then $N_G(v)$ denotes the set of its neighbors, and $|N_G(v)|$ is the degree of $v$, which is 
denoted by $d_G(v)$. 
The maximum degree and the minimum degree of a vertex of $G$ are denoted by $\Delta(G)$ and $\delta(G)$, respectively. 
For $i \in \{1, \dots,\Delta(G)\}$ let $V_i(G) = \{v : d_G(v)=i\}$.
 
A $k$-edge-coloring of $G$ is a function $\phi: E(G) \longrightarrow \{1, \dots,k\}$
such that $\phi(e) \not= \phi(f)$ for adjacent edges $e$ and $f$.  
The chromatic index $\chi'(G)$ is the smallest number $k$ such that there is $k$-coloring of $G$.  
In 1965 Vizing proved the fundamental result on the chromatic index of simple graphs.

\begin{theorem} [\cite{Vizing_1965_2}] \label{Vizing}
If $G$ is a graph, then $\chi'(G) \in \{\Delta(G), \Delta(G)+1\}$. 
\end{theorem}

Theorem \ref{Vizing} leads to a natural classification of simple graphs into two classes, namely 
Class $1$ and Class $2$ graphs depending upon whether their edge chromatic number is $\Delta$ and $\Delta + 1$.
For $k \geq 2$, a graph $G$ is $k$-critical if $\Delta(G) = k$, $\chi'(G) = k+1$ and $\chi'(G-e) = k$ for every $e \in E(G)$. 
Let ${\cal{C}}(k)$ be the set of $k$-critical graphs, and ${\cal{C}} = \bigcup_{k=2}^{\infty}{\cal{C}}(k)$
be the set of critical graphs. 

If $G$ is a graph, then $\alpha(G)$ denotes the maximum cardinality of an independent set of vertices in $G$.
The independence ratio of $G$ is $\frac{\alpha(G)}{|V(G)|}$ and it is denoted by $\iota(G)$. 
In 1965, Vizing \cite{Vizing_1965} conjectured that the independence ratio of edge-chromatic critical graphs is at most $\frac{1}{2}$.

\begin{conjecture} [\cite{Vizing_1965}] \label{INC} 
If $G \in {\cal{C}}$, then $\iota(G) \leq \frac{1}{2}$.
\end{conjecture}

Clearly, Conjecture \ref{INC} can be reformulated as follows.

\begin{conjecture}  [\cite{Vizing_1965}] \label{INC2} 
For all $k \geq 2$, if $G \in {\cal{C}}(k)$, then $\iota(G) \leq \frac{1}{2}$.
\end{conjecture}

Since the 2-critical graphs are the odd circuits, it follows that Conjecture \ref{INC2} is true for $k=2$. 
It is an open question whether it is true for $k \geq 3$.
It is easy to see, that the bound $1/2$ cannot be replaced by a smaller one. The first results on this topic were obtained by 
Brinkmann et al.~\cite{BCGS_2000} who proved that the independence ratio of critical graphs is smaller than $\frac{2}{3}$. 
In \cite{Stefan_es_2004} Conjecture \ref{INC} is verified for overfull graphs, i.e.~graphs $G$ with  $|E(G)| > \Delta(G) \lfloor \frac{|V(G)|}{2} \rfloor$. 
In 2006, Luo and Zhao \cite{Luo_Zhao_2006} proved that the conjecture is true for critical graphs whose
order is at most twice the maximum degree of the graph. Later some improvements were achieved for specific values of $\Delta$, see 
\cite{Luo_Zhao_2006, Luo_Zhao_2009, Luo_Zhao_2011, Miao_2011, Miao_etal_2015}. 
In 2011, Woodall \cite{Woodall_2011} completed a major step in this research by proving that the independence ratio of 
critical graphs is bounded by $\frac{3}{5}$.

The main result of this article is that for each $\epsilon > 0$, Conjecture \ref{INC} is equivalent to its restriction on a specific set ${\cal{C}}_{\epsilon}$
of critical graphs and $\iota(G) < \frac{1}{2} + \epsilon$ for each $G \in {\cal{C}}_{\epsilon}$. For the proof of this statement 
we will deduce similar results for ${\cal{C}}(k)$, for each $k \geq 3$.

\section{$k$-critical graphs and Meredith extension} 

This section first studies $k$-critical graphs and Conjecture \ref{INC2}.  
One of the fundamental statements in the theory of edge-coloring of graphs is Vizing's Adjacency Lemma. 

\begin{lemma} [Vizing's Adjacency Lemma \cite{Vizing_1965_2}] \label{VAL}
Let $G$ be a critical graph. If $xy \in E(G)$, then at least $\Delta(G)-d_G(y)+1$ vertices in $N_G(x)\setminus \{y\}$ have degree  $\Delta(G)$.
\end{lemma}

Lemma \ref{VAL} implies that if $v$ is a vertex of a $k$-critical graph, then it is adjacent to at least two vertices of degree $k$. 

\begin{definition} \label{Def}
For $k \geq 2$ and $t \geq 0$ let ${\cal{C}}(k,t)$ be the set of $k$-critical graphs $G$ with the following properties: 

\begin{enumerate}
\item $\delta(G) \geq k-1$.
\item every $v \in V_{k-1}(G)$ is the initial vertex of $k-1$ distinguished paths $p^t_1(v), \dots ,p^t_{k-1}(v)$ such that for all $i,j \in \{1,\dots,k-1\}$:
		\begin{enumerate}
		\item $V(p^t_i(v)) \cap V_{k-1}(G) = \{v\}$,
		\item $|V(p^t_i(v))| \geq 2t(k-1)+2$, 
		\item if $i \not = j$, then $V(p^t_i(v)) \cap V(p^t_j(v)) = \{v\}$, and		
		\item if $w \in V_{k-1}(G)$ and $w\not=v$, then $V(p^t_i(v)) \cap V(p^t_j(w)) = \emptyset$.
		\end{enumerate} 
\end{enumerate}
\end{definition}

For $k\geq 0$ and $t \geq 0$, let $\iota(k) = \sup\{\iota(G) : G \in {\cal{C}}(k)\}$ and $\iota(k,t) = \sup\{\iota(G) : G \in {\cal{C}}(k,t)\}$.
We will prove that for any $k \geq 3$ and any $t\geq 0$, Conjecture \ref{INC2} for ${\cal{C}}(k)$ is equivalent to 
its restriction on ${\cal{C}}(k,t)$. We prove upper bounds for $\iota(k,t)$ and $\lim_{t \rightarrow \infty} \iota(k,t) = \frac{1}{2}$.
These statements are used to deduce the main result of this article. 

The 2-critical graphs are the odd circuits and for any $k \geq 2$, there exists a
$k$-critical graph $G$ with $\delta(G)=2$. Hence, the following lemma is an obvious consequence of Lemma \ref{VAL} and Definition \ref{Def}.

\begin{proposition} \label{Motiv}
\begin{enumerate}
	\item ${\cal{C}}(3,0) = {\cal{C}}(3)$ and ${\cal{C}}(2,t) = {\cal{C}}(2)$ for all $t \geq 0$. 
	\item If $k \geq 2$ and $t \geq 0$, then ${\cal{C}}(k,t+1) \subseteq {\cal{C}}(k,t) \subseteq {\cal{C}}(k)$.
\end{enumerate}
\end{proposition}

The following operation on graphs was first studied by Meredith \cite{Meredith_1973}.

\begin{definition}Let $k \geq 2$ and $G$ be a graph with $\Delta(G)=k$, $v\in V(G)$ with $d_{G}(v)=d$, and  
let $v_{1},\dots,v_{d}$ be the neighbors of $v$. Let $u_{1},\dots,u_{k}$ be the vertices of degree $k-1$
in a complete bipartite graph $K_{k,k-1}$. The graph $H$ is a Meredith extension of $G$ (applied on $v$) 
if it is obtained from  $G-v$ and $K_{k,k-1}$ by adding edges $v_{i}u_{i}$ for each $i \in \{1,...,d\}$.
\end{definition}

The following theorem is Theorem 2.1 in \cite{Eckhard_1999}.

\begin{theorem} [\cite{Eckhard_1999}] \label{T12}
Let $k \geq 2$, $G$ be a graph with $\Delta(G)=k$ and $M$ be a Meredith extension of $G$. Then 
$G$ is $k$-critical if and only if $M$ is $k$-critical.
\end{theorem}

\begin{lemma} \label{extension_equiv}
Let $k \geq 2$, $G$ be a graph with $\Delta(G)=k$ and $H$ be a Meredith extension of $G$. 
Then $\iota(G) \leq \frac{1}{2}$ if and only if $\iota(H) \leq \frac{1}{2}$.
\end{lemma}

\begin{proof} We prove $\iota(G) > \frac{1}{2}$ if and only if $\iota(H) > \frac{1}{2}$.

Let $v \in V(G)$ and $H$ be the Meredith extension of $G$ applied on $v$. 
We have $|V(H)|=|V(G)|+2k-2$ and hence $|V(H)|$ and $|V(G)|$ have the same parity. 

($\Rightarrow$) Let $I_G$ be an independent set of $G$ with more than $\frac{1}{2}|V(G)|$ vertices.
 
If $v \in I_G$, then all neighbors of $v$ are not in $I_G$. Hence, $H$ has an independent set $I_{H}$ of cardinality $|I_G|-1+k$. 
Therefore,  $|I_{H}| = |I_G|+k-1 > \frac{1}{2}(|V(G)| + 2k - 2) = \frac{1}{2}|V(H)|$. 
	
If $v \not \in I_G$, then $H$ has an independent set $I_{H}$ of cardinality $|I_G| + (k-1)$, e.g.~$I_G \cup V_k(K_{k,k-1})$.
We deduce $|I_H| > \frac{1}{2}|V(H)|$ as above.

($\Leftarrow$) Let $I_H$ be an independent set of $H$ with $|I_H| > \frac{1}{2} |V(H)|$. We can assume that 
$I_H$ is maximum. Let 
$K_{k,k-1}$ be the subgraph of $H$ which was added to $G-v$ by applying Meredith extension on $v$.

If there is a vertex $w \in V_{k-1}(K_{k,k-1})$ which has a neighbor in $(V(H) - V(K_{k,k-1})) \cap I_H$, then 
$|V(K_{k,k-1}) \cap I_H|=k-1$. Hence, if we contract $K_{k,k-1}$ to a single vertex $v$ (to obtain $G$),
then $I_G = I_H - V(K_{k,k-1})$ is an independent set in $G$ which contains $|I_H|-(k-1)$ vertices.
Hence $|I_G|= |I_H|-(k-1) > \frac{1}{2}(|V(H)| - (2k - 2)) = \frac{1}{2}|V(G)|$.  

If for every vertex $w \in V_{k-1}(K_{k,k-1})$ all neighbors in $H - V(K_{k,k-1})$ are not in $I_H$,
then $|V(K_{k,k-1}) \cap I_H|=k$. If we contract $K_{k,k-1}$ to a single vertex $v$, then
$I_G = (I_H - V(K_{k,k-1})) \cup \{v\}$ is an independent set in $G$. As above, we deduce that 
$|I_G| > \frac{1}{2}|V(G)|$. 
\end{proof}

\begin{lemma} \label{ext_to_*}
For every $k \geq 2$ and every $t \geq 0$: Every $k$-critical graph $G$ can be extended to a graph $H \in {\cal{C}}(k,t)$ by a sequence of Meredith extensions.
\end{lemma}

\begin{proof} For $k=2$ there is nothing to prove. Let $k \geq 3$.
We first show that $G$ can be extended to a graph of ${\cal{C}}(k,0)$.
If $G \in {\cal{C}}(k,0)$, then we are done. Assume that $G \in {\cal{C}}(k) \setminus {\cal{C}}(k,0)$.
We proceed in three steps. For an example see Figures \ref{K_4_4_plus}, \ref{K_4_4_plus_ext_1} and \ref{K_4_4_plus_ext_2} (without step 2).  

(1) Repeated application of Meredith extension on all vertices of degree smaller than $k-1$, yields a graph $G_1$ 
    with $d_{G_1}(v) \in \{k-1, k\}$, for all $v \in V(G_1)$. 
		
(2) Repeated application of Meredith extension on vertices of degree $k-1$ which are adjacent to another vertex of degree $k-1$,
yields a graph $G_2$, with $d_{G_2}(v) \in \{k-1, k\}$, for all $v \in V(G_2)$, and $V_{k-1}(G_2)$ is an independent set. 

(3) Repeated application of Meredith extension on vertices of degree $k-1$ which have a common neighbor yields a graph $G_3$ 
with $d_{G_3}(v) \in \{k-1, k\}$, 
$V_{k-1}(G_3)$ is an independent set, and $N_{G_3}(u) \cap N_{G_3}(w) = \emptyset$ for any two vertices $u, w \in V_{k-1}(G_3)$.

Let $H=G_3$. By Theorem \ref{T12}, $H$ is $k$-critical and it satisfies the conditions of Definition \ref{Def} for $t=0$. 
Hence, $H \in {\cal{C}}(k,0)$.

Next we show that every graph $G'$ of ${\cal{C}}(k,s)$ $(s \geq 0)$ can be extended to a graph $H'$ of ${\cal{C}}(k,s+1)$ by 
a sequence of Meredith extensions. 
Let $v \in V_{k-1}(G')$ and $p^s_j(v)$ be one of the $k-1$ distinguished paths which have $v$ as initial vertex. 
Let $z$ be the terminal vertex of $p^s_j(v)$. Apply Meredith extension on $z$
and extend $p^s_j(v)-z$ to a path $p^{s+1}_j(v)$ that contains all vertices of the $K_{k,k-1}$ which is used in the Meredith extension. 
Then $|V(p^{s+1}_j(v))| = |V(p^{s}_j(v))| + 2k -2 \geq 2s(k-1)+2 + 2k - 2 = 2(s+1)(k-1) + 2$. If we repeat this procedure on all terminal vertices of the 
distinguished paths of $G'$ we obtain a graph $H' \in {\cal{C}}(k,s+1)$.
\end{proof}

\begin{figure}[ht!] 
\centering
\includegraphics[scale=.5]{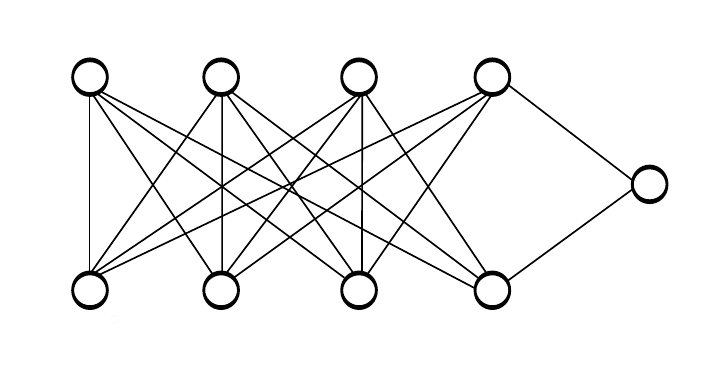}
\caption{Graph $H \in {\cal{C}}(4)$} \label{K_4_4_plus}
\end{figure}

\begin{figure}[ht!] 
\centering
\includegraphics[scale=.5]{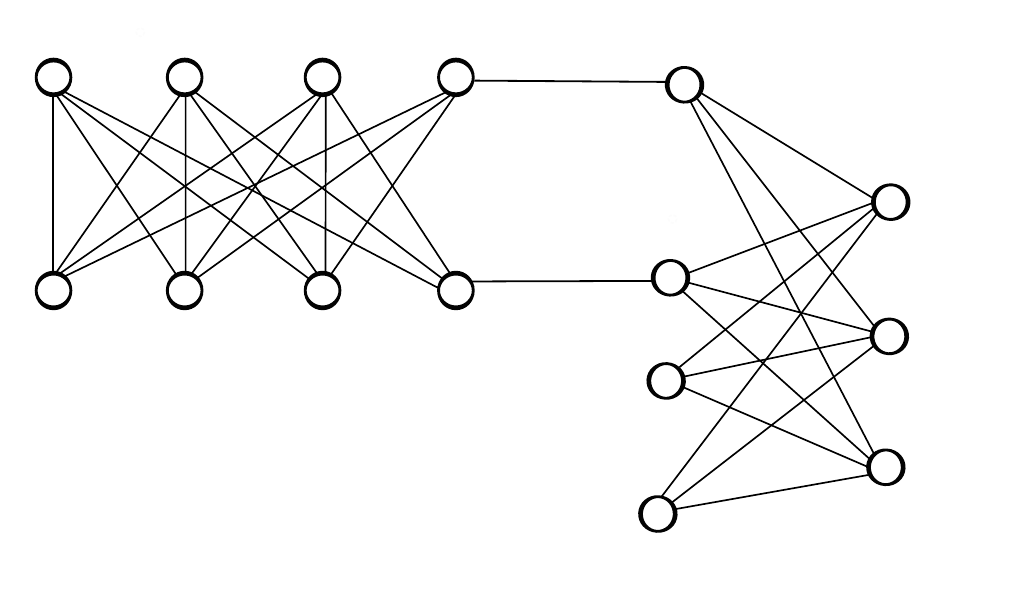}
\caption{Graph $H' \in {\cal{C}}(4)$ (Step 1)} \label{K_4_4_plus_ext_1}
\end{figure}

\begin{figure}[ht!] 
\centering
\includegraphics[scale=.5]{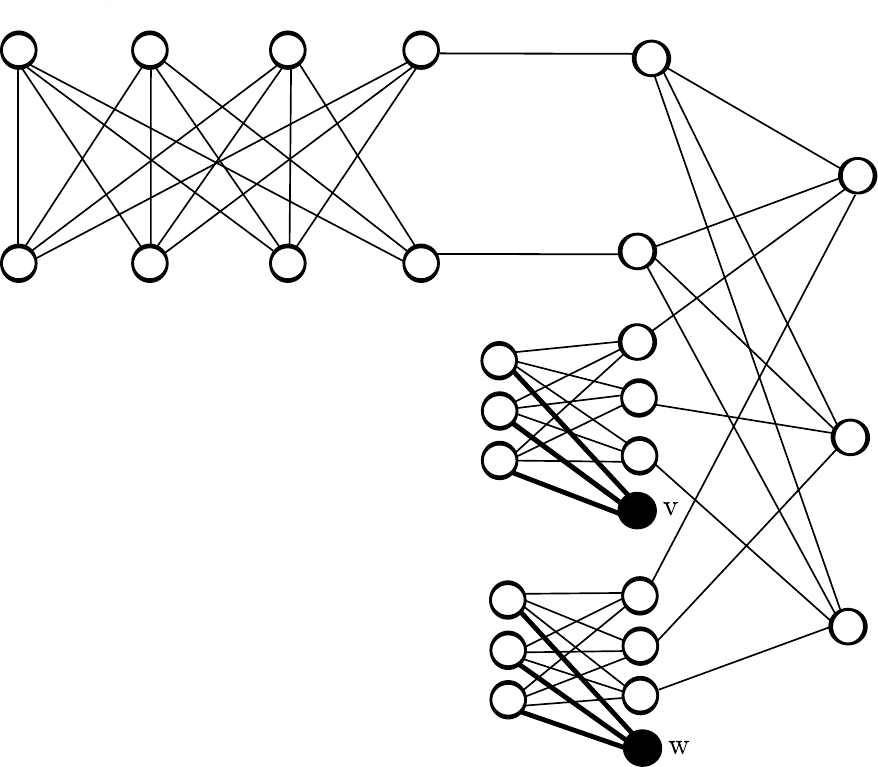}
\caption{Graph $H_0 \in {\cal{C}}(4,0)$ (Step 3)} \label{K_4_4_plus_ext_2}
\end{figure}

The notation in Figures \ref{K_4_4_plus}, \ref{K_4_4_plus_ext_1} and \ref{K_4_4_plus_ext_2} are used in the proof of Theorem \ref{limits}.
For $i \in \{1,2,3\}$, the paths $p^{0}_i(v)$ and $p^{0}_i(w)$ are indicated by the bold edges. 
The following lemma is obvious. 

\begin{lemma} \label{closed} Let $k \geq 2$, $t \geq 0$ and $G \in {\cal{C}}(k,t)$. If $H$ is a Meredith extension of $G$, then $H \in {\cal{C}}(k,t)$.
\end{lemma}

\begin{theorem} \label{equiv}
For every $k \geq 2$ and every $t \geq 0$: $\iota(k) \leq \frac{1}{2}$ if and only if $\iota(k,t) \leq \frac{1}{2}$.
\end{theorem}

\begin{proof}
By Proposition \ref{Motiv},  ${\cal{C}}(k,t) \subseteq {\cal{C}}(k)$ for all $k \geq 2$ and $t \geq 0$. Hence, if 
$\iota(k) \leq \frac{1}{2}$ then $\iota(k,t) \leq \frac{1}{2}$.

Let $G \in {\cal{C}}(k)$. If there is $t ' \geq t$ such that $G \in {\cal{C}}(k,t')$, then we are done,
since ${\cal{C}}(k,t') \subseteq {\cal{C}}(k,t)$ by Proposition \ref{Motiv}.
If $G \not \in {\cal{C}}(k,t')$ for all $t' \geq t$, then it follows with Lemma \ref{ext_to_*} 
that there exists $H \in {\cal{C}}(k,t)$ 
which is obtained from $G$ by a sequence of Meredith extensions. 
By our assumption, $\iota(H) \leq \frac{1}{2}$ and hence, $\iota(G) \leq \frac{1}{2}$
by Lemma \ref{extension_equiv}. Therefore, $\iota(k) \leq \frac{1}{2}$.
\end{proof}

\begin{theorem} \label{upper_bound} Let $k \geq 2$, $t \geq 0$ and $\varphi(k,t) = t(k-1)^2 + k-1$.
If $G \in {\cal{C}}(k,t)$, then $\iota(G) < \frac{1}{2} + \frac{1}{4k \varphi(k,t) +2}$.
\end{theorem}

\begin{proof} If $G \in {\cal{C}}(2)$, then $\iota(G) < \frac{1}{2}$. Let $G \in {\cal{C}}(k,t)$ ($k \geq 3$, $t \geq 0$) and
$I$ be an independent set of $G$ and $Y = V(G)-I$.
Let $I_k = I \cap V_k(G)$, $I_{k-1} = I \cap V_{k-1}(G)$, $Y_k = Y \cap V_k(G)$, $Y_{k-1} = Y \cap V_{k-1}(G)$.

Clearly, $I$ contains vertices of $V_{k-1}(G)$. Let $v$ be such a vertex. By definition, there are $k-1$ distinguished paths $p^t_1(v), \dots ,p^t_{k-1}(v)$ 
such that for all $i,j \in \{1,\dots,k-1\}$

\begin{enumerate}
		\item[(a)] $V(p^t_i(v)) \cap V_{k-1}(G) = \{v\}$,
		\item[(b)] $|V(p^t_i(v))| \geq 2t(k-1)+2$, 
		\item[(c)] if $i \not = j$, then $V(p^t_i(v)) \cap V(p^t_j(v)) = \{v\}$, and		
		\item[(d)] if $w \in V_{k-1}(G)$ and $w\not=v$, then $V(p^t_i(v)) \cap V(p^t_j(w)) = \emptyset$.
\end{enumerate}

Consequently, $|Y \cap V(p^t_i(v))| \geq t(k-1) + 1$ for each $i \in \{1, \dots,k-1\}$, and therefore $\varphi(k,t)|I_{k-1}| \leq |Y|$. 
Let $m_Y = |E(G[Y])|$. Since $G$ is a critical graph it follows that $m_Y > 0$. With $|I_{k-1}| \leq \frac{1}{\varphi(k,t)}|Y|$ we deduce   

$$k|I| - \frac{1}{\varphi(k,t)}|Y| \leq k|I| - |I_{k-1}| \leq k|Y| - 2m_Y < k|Y|. $$

Since $Y = V(G)-I$, it follows that

			$$|I| <  \frac{k + \frac{1}{\varphi(k,t)}}{2k + \frac{1}{\varphi(k,t)}}|V(G)|. $$
			
Therefore, $\iota(G) < \frac{1}{2} + \frac{1}{4k \varphi(k,t) +2}$
\end{proof}

We now deduce our main results. 

\newpage

\begin{theorem} \label{limits} 
For each $k \geq 2$: $\lim_{t \rightarrow \infty} \iota(k,t) = \frac{1}{2}$.
\end{theorem}
\begin{proof}

The statement is trivial for $k=2$. We will first prove the following claim.

\begin{claim} \label{claim}
For all $k \geq 3$ and $t \geq 0$: $\iota(k,t) \geq \frac{1}{2}$.
\end{claim}
  
We show that for every $\epsilon > 0$ and all $k \geq 3$ and $t \geq 0$ the set ${\cal{C}}(k,t)$ contains 
a graph $G$ with $i(G) > \frac{1}{2} - \epsilon$. 

Let $H$ be the graph which is obtained from the complete bipartite graph $K_{k,k}$ by subdividing one edge. 
It is easy to see that $H$ is $k$-critical. Let $H'$ be the graph obtained from $H$ by applying 
Meredith extension on the divalent vertex of $H$ and let $H_0$ be the graph obtained from $H'$ by applying
Meredith extension on all vertices of $V_{k-1}(H')$. Hence, $H_0 \in {\cal{C}}(k,0)$. To obtain a graph $H_t$ of ${\cal{C}}(k,t)$ ($t \geq 1$)
apply Meredith extension on the terminal vertices of the distinguished paths of $H_{t-1}$ as described in the proof of Lemma \ref{ext_to_*}. 
Starting with $H_t=H^0_t$, construct an infinite sequence $H^0_t, H^1_t\dots$ of graphs by Meredith extension. By Lemma \ref{closed}, these graphs
are in ${\cal{C}}(k,t)$.

If $H^{i}_t$ is obtained from $H$ by applying Meredith extension $n_i$ times, then $|V(H^{i}_t)| = 2(k + n_ik -n_i)+1$ and it has an independent set of 
$k + n_ik - n_i$ vertices. Hence, $\alpha(H) \geq \frac{1}{2} - \frac{1}{2(2k+2n_i(k-1) + 1)}$. Choose $n_i$ such that 
$2k+2n_i(k-1) + 1 > \frac{1}{2\epsilon}$ and the claim is proved. 

By Theorem \ref{upper_bound}, we have $\iota(k,t) \leq  \frac{1}{2} + \frac{1}{4k \varphi(k,t) +2}$, where $\varphi(k,t) = t(k-1)^2 + k-1$. 
Since $\varphi(k,t+1) > \varphi(k,t)$ 
it follows with the Claim \ref{claim} that $\lim_{t \rightarrow \infty} \iota(k,t) = \frac{1}{2}$. 
\end{proof}

\begin{theorem} \label{main}
For every $\epsilon > 0$, there is a set $\cal{C}_{\epsilon}$ of critical graphs such that 
\begin{enumerate}
	\item $\iota(G) \leq \frac{1}{2}$ for every $G \in {\cal{C}}$ if and only if $\iota(G) \leq \frac{1}{2}$ for every $G \in {\cal{C}_{\epsilon}}$.
	\item If $G \in {\cal{C}_{\epsilon}}$, then $\iota(G) < \frac{1}{2} + \epsilon$. 
\end{enumerate}
\end{theorem}

\begin{proof} Let $\epsilon > 0$ be given. We first construct $\cal{C}_{\epsilon}$. Let $\varphi(k,t)= t(k-1)^2 + k-1$ and  
for $k = 3$ choose $t_3 \geq 0$ such that $\frac{1}{4k \varphi(k,t_3) +2} = \frac{1}{12 \varphi(3,t_3) +2} < \epsilon$.
Let ${\cal{C}_{\epsilon}} = \bigcup_{k=2}^{\infty} {\cal{C}}(k,t_3)$.

We have ${\cal{C}} = \bigcup_{k=2}^{\infty}{\cal{C}}(k)$.
For $k \geq 2$ it follows with Theorem \ref{equiv} that
$\iota(G) \leq \frac{1}{2}$ for every $G \in {\cal{C}}(k)$ if and only if $\iota(G) \leq \frac{1}{2}$ for every $G \in {\cal{C}}(k,t_3)$.
Therefore, 
$\iota(G) \leq \frac{1}{2}$ for every $G \in {\cal{C}}$ if and only if $\iota(G) \leq \frac{1}{2}$ for every $G \in {\cal{C}_{\epsilon}}$.

It remains to prove statement 2. Let $G \in {\cal{C}_{\epsilon}}$. If $G \in {\cal{C}}(2)$, then $\iota(G) < \frac{1}{2}$.
Let $k \geq 3$ and $G \in {\cal{C}}(k,t_3)$.
We have $\varphi(k+1,t) > \varphi(k,t)$ and thus, $\frac{1}{4k \varphi(k,t_3) +2} \leq \frac{1}{12 \varphi(3,t_3) +2} < \epsilon$.
It follows with Theorem \ref{upper_bound} that $\iota(G) < \frac{1}{2} + \frac{1}{4k \varphi(k,t_3) +2} < \frac{1}{2} + \epsilon$. 
Therefore, if $G \in {\cal{C}_{\epsilon}}$, then $\iota(G) < \frac{1}{2} + \epsilon$. 
\end{proof}

\subsection*{Concluding remark}

Let $s \in \{1, \dots, k-1\}$.
The main results (Theorems \ref{limits} and \ref{main}) can also be deduced if we ask for the existence of $s$ distinguished 
paths in Definition \ref{Def}, say to define ${\cal{C}}_s(k,t)$. If we change $\varphi(k,t)$ in Theorem \ref{upper_bound} to $\varphi_s(k,t) = st(k-1)+s$,
then we similarly can deduce that if $G \in {\cal{C}}_s(k,t)$, then $\iota(G) < \frac{1}{2} + \frac{1}{4k \varphi_s(k,t) +2}$. The two natural 
choices for $s$ are 1 and $k-1$. We took $k-1$ since then the structural properties of $3$-critical graphs which are implied by 
Vizing's Adjacency Lemma are generalized to graphs of ${\cal{C}}(k,0)$.

\bibliographystyle{acm}
\bibliography{G_FCG_AJC}

\end{document}